\documentclass[12pt]{amsart}

\usepackage{amsmath,amsthm,amssymb}
\usepackage[top=30truemm,bottom=30truemm,left=30truemm,right=30truemm]{geometry}
\usepackage[all]{xy}

\usepackage[dvipsnames]{xcolor}
\usepackage{iftex}
\ifPDFTeX
  \usepackage[linktocpage]{hyperref}
\else
  \usepackage[dvipdfmx,linktocpage]{hyperref}
\fi

\definecolor{mycolor}{rgb}{1.0,0,0.2}
\hypersetup{
    colorlinks=true,
    citecolor=mycolor,
    linkcolor=mycolor,
}

\usepackage{aliascnt} 

\newtheorem{theorem}{Theorem}[section]

\newaliascnt{proposition}{theorem}
\newtheorem{proposition}[proposition]{Proposition}
\aliascntresetthe{proposition}

\newaliascnt{lemma}{theorem}
\newtheorem{lemma}[lemma]{Lemma}
\aliascntresetthe{lemma}

\newaliascnt{corollary}{theorem}
\newtheorem{corollary}[corollary]{Corollary}
\aliascntresetthe{corollary}

\newaliascnt{definition}{theorem}

\aliascntresetthe{definition}

\newaliascnt{remark}{theorem}
\newtheorem{remark}[remark]{Remark}
\aliascntresetthe{remark}

\newaliascnt{example}{theorem}
\newtheorem{example}[example]{Example}
\aliascntresetthe{example}

\numberwithin{equation}{section}

\newcommand{\rank}{\textrm{rank}}
\newcommand{\vol}{\textrm{vol}}

\newcommand{\End}{$\hfill{\blacksquare}$}
\newcommand{\partialt}{\displaystyle\frac{\partial}{\partial t}}

\newcommand{\prarrow}[2]{\ar@<0.5ex>[r]^-{#1} \ar@<-0.5ex>[r]_-{#2}}
\newcommand{\plarrow}[2]{\ar@<0.5ex>[l]^-{#1} \ar@<-0.5ex>[l]_-{#2}} 
\newcommand{\pdarrow}[2]{\ar@<0.5ex>[d]^-{#1} \ar@<-0.5ex>[d]_-{#2}} 
\newcommand{\puarrow}[2]{\ar@<0.5ex>[u]^-{#1} \ar@<-0.5ex>[u]_-{#2}}

\makeatletter
\newcommand{\subscripts}[3]{%
  \@mathmeasure\z@\displaystyle{#2}%
  \global\setbox\@ne\vbox to\ht\z@{}\dp\@ne\dp\z@
  \setbox\tw@\box\@ne
  \@mathmeasure4\displaystyle{\copy\tw@_{#1}}%
  \@mathmeasure6\displaystyle{{#2}_{#3}}%
  \dimen@-\wd6 \advance\dimen@\wd4 \advance\dimen@\wd\z@
  \hbox to\dimen@{}\mathop{\kern-\dimen@\box4\box6}%
}
\makeatother

\begin{document}

\title{Godbillon-Vey classes of regular Jacobi manifolds}
\author{Shuhei Yonehara}
\address{National Institute of Technology, Yonago College, Tottori, 683-8502, JAPAN}
\email{shuhei.yonehara.0201@gmail.com}

\begin{abstract}
  The notion of a Jacobi manifold is a natural generalization of that of a Poisson manifold. A Jacobi manifold has a natural foliation in which each leaf has either a contact structure or a locally conformal symplectic structure. In this paper, we study a characteristic class called the Godbillon-Vey class for Jacobi manifolds with regular foliation, and express it explicitly in terms of Jacobi structures.
\end{abstract}

\maketitle

\tableofcontents

\footnote[0]{2020 \textit{Mathematics Subject Classification}\ : 53D17, 53D10, 53C12}

\section{Introduction}

The notion of a \textit{Jacobi manifold} \cite{lichnerowicz1978varietes} is a natural generalization of that of a Poisson manifold. A Jacobi structure on a manifold is defined by a pair of a vector field $E$ and a bivector field $\pi$ such that \[[\pi,\pi]_S=2E\wedge\pi,\qquad [\pi,E]_S=0,\]
where $[\cdot,\cdot]_S$ denotes the \textit{Schouten bracket}. A Jacobi manifold $(M,\pi,E)$ possesses a foliation $\mathcal{F}_{(\pi,E)}$ defined by a distribution $T\mathcal{F}_{(\pi,E)}=\textrm{Im}\pi^\sharp+\mathbb{R}\langle E\rangle$ where the even-dimensional leaves of the foliation have locally conformal symplectic structures (hereafter referred to as LCS structures), and the odd-dimensional leaves have contact structures. A Jacobi manifold is said to be regular if the foliation is regular, that is, in the cases that the manifold has a regular LCS foliation or a regular contact foliation. In this paper, we refer to the former case as an LCS-type regular Jacobi manifold and the latter case as a contact-type regular Jacobi manifold.

Godbillon and Vey \cite{godbillon1971invariant} defined a characteristic class of regular co-orientable foliations, that is called the \textit{Godbillon-Vey class}. This class is defined as follows: first, for a regular co-orientable foliation $\mathcal{F}$ of codimension $q$, we take a $q$-form $\alpha$ defining the distribution $T\mathcal{F}$. Since the distribution is integrable, there is a 1-form $\beta$ which satisfies $d\alpha=\beta\wedge\alpha$. Then the form $\beta\wedge(d\beta)^q$ is closed. Moreover, the de Rham cohomolory class $[\beta\wedge(d\beta)^q]\in H^{2q+1}_{dR}(M)$ is independent of the choice of the forms $\alpha$ and $\beta$.

Mikami \cite{mikami2000godbillon} explicitly described the Godbillon-Vey classes of the symplectic foliations on regular Poisson manifolds. In this paper, we extend the result to the case of Jacobi manifolds. Specifically, we describe the Godbillon-Vey classes of regular Jacobi manifolds in terms of the tensors $(\pi, E)$. The following is our main result:

\begin{theorem}[\autoref{thm:100},\ \autoref{thm:101}]
  Let $(M,\pi,E)$ be an $n$-dimensional orientable regular Jacobi manifold and fix a volume form ${\rm{vol}}_M$ on $M$. We define an isomorphism between vector bundles $\varphi:\wedge^k TM\to\wedge^{n-k}T^\ast M$ by $\varphi(X)=\iota_X{\rm{vol}}_M$ for each $0\leq k\leq n$. For any $U\in\mathfrak{X}^k(M)$, let $\ast U\in\mathfrak{X}^{n-k}$ be one of multivector fields which satisfies ${\rm vol}_M(U\wedge\ast U)=1$. Then we have the following.
  \begin{enumerate}
    \item For an LCS-type regular Jacobi manifold $(M,\pi,E)$ such that ${\rm rank}T\mathcal{F}_{(\pi,E)}=2m$, we can take
    \[\alpha=\varphi\left(\frac{\pi^m}{m!}\right),\qquad\beta=\varphi[-\ast\pi^m,\pi^m]_S\]
    (so $\mathcal{F}_{(\pi,E)}$ is co-orientable), and the Godbillon-Vey class of $\mathcal{F}_{(\pi,E)}$ is given by 
    \[\big[\varphi(\iota_{(d\varphi[-\ast\pi^m,\pi^m]_S)^{n-2m}}[-\ast\pi^m,\pi^m]_S)\big].\]

    \item For a contact-type regular Jacobi manifold $(M,\pi,E)$ such that ${\rm rank}T\mathcal{F}_{(\pi,E)}=2m+1$, we can take 
    \[\alpha=\varphi\left(\frac{\pi^m}{m!}\wedge E\right),\qquad\beta=\varphi[\ast(\pi^m\wedge E),\pi^m\wedge E]_S\]
    (so $\mathcal{F}_{(\pi,E)}$ is co-orientable), and the Godbillon-Vey class of $\mathcal{F}_{(\pi,E)}$ is given by
    \[\big[\varphi(\iota_{(d\varphi[\ast(\pi^m\wedge E),\pi^m\wedge E]_S)^{n-2m-1}}[\ast(\pi^m\wedge E),\pi^m\wedge E]_S)\big].\]\hfill{$\Box$}
  \end{enumerate}
\end{theorem}

In this paper, we prove the above results based on the method in \cite{mikami2000godbillon} and provide an alternative proof for contact-type regular Jacobi manifolds using Poissonization of them. Furthermore, as specific examples, we examine the Godbillon-Vey classes in detail for the codimension one case.

\vspace{0.2in}
\noindent{\bf Acknowledgments.} This work was supported by JSPS KAKENHI Grant Number JP23KJ1487.

\section{Jacobi manifolds}

A \textit{Jacobi structure} on a manifold $M$ is a bracket $\{\cdot,\cdot\}:C^\infty(M)\times C^\infty(M)\to C^\infty(M)$ which satisfies the following conditions:
\begin{enumerate}
  \item $\{\cdot,\cdot\}$ is $\mathbb{R}$-bilinear, skew-symmetric,
  \vspace{0.03in}
  \item $\{\{f,g\},h\}+\{\{g,h\},f\}+\{\{h,f\},g\}=0\qquad\text{(Jacobi identity),}$
  \vspace{0.03in}
  \item $\textrm{supp}\{f,g\}\subset \textrm{supp}(f)\cap\textrm{supp}(g)$.
\end{enumerate}

\vspace{0.1in}
We denote the set of all $k$-forms on a manifold $M$ by $\Omega^k(M)$, and the set of all $k$-vectors on $M$ by $\mathfrak{X}^k(M)$, where $0\leq k\leq\dim M$.

\vspace{0.1in}
A Jacobi structure on $M$ is also characterized by a pair $(\pi,E)$ of a bivector $\pi\in\mathfrak{X}^2(M)$ and a vector field $E\in\mathfrak{X}(M)$ which satisfies
\[[\pi,\pi]_S=2E\wedge\pi,\qquad [\pi,E]_S=0,\]
where $[\cdot,\cdot]_S:\mathfrak{X}^k(M)\times\mathfrak{X}^l(M)\to\mathfrak{X}^{k+l-1}(M)$ denotes the \textit{Schouten bracket}, which is defined by

\[[X_1\wedge\cdots\wedge X_k,Y_1\wedge\cdots\wedge Y_l]_S=\displaystyle\sum_{i,j}(-1)^{i+j}[X_i,Y_j]\wedge X_1\wedge\cdots\wedge \check{X_{i}}\wedge\cdots\wedge X_k\wedge Y_1\wedge\cdots\wedge \check{Y_{j}}\wedge\cdots\wedge Y_l.\]
For $X\in\mathfrak{X}^1(M)=\mathfrak{X}(M)$ and $f,g\in\mathfrak{X}^0(M)=C^\infty(M)$, we define $[X,f]_S=\mathcal{L}_Xf=Xf$ and $[f,g]_S=0$. This bracket has the following property:
\begin{equation}\label{eq:-3}
  [U,V]_S=-(-1)^{(k-1)(l-1)}[V,U]_S,
\end{equation}
\begin{equation}\label{eq:-2}
  [U,V\wedge W]_S=[U,V]_S\wedge W+(-1)^{(k-1)l}V\wedge[U,W]_S
\end{equation}
for $U\in\mathfrak{X}^k(M),V\in\mathfrak{X}^l(M)$ and $W\in\mathfrak{X}^\bullet(M)$.

\vspace{0.1in}
Two definitions of Jacobi structures are related by the formula
\[\{f,g\}=\pi(df,dg)+f(Eg)-g(Ef).\]
In the case of $E=0$, the Jacobi structure turn to be a Poisson structure and the corresponding bracket satisfies the Leibniz identity
\[\{f,gh\}=\{f,g\}h+g\{f,h\}.\]

\begin{example}[Contact manifolds]
  A \textit{contact structure} on a $(2n+1)$-dimensional manifold $M$ is a 1-form $\theta$ which satisfies $\theta\wedge(d\theta)^n\neq0$.
  On a contact manifold $(M,\theta)$, there is a unique vector field $E$ (called the \textit{Reeb vector field}) which satisfies 
  \[d\theta(E,-)=0,\qquad\theta(E)=1.\]
  Moreover, we can define an isomorphism $\flat:TM\to T^\ast M$ by
  \[\flat(X)=d\theta(X,-)+\theta(X)\theta(-)\]
  and a bivector $\pi$ by
  \[\pi(\alpha,\beta)=d\theta(\flat^{-1}(\alpha),\flat^{-1}(\beta)).\]
  Then the pair $(\pi,E)$ is a Jacobi structure on $M$.\End
\end{example}

\begin{example}[Locally conformal symplectic (LCS) manifolds]
  A \textit{locally conformal symplectic (LCS) structure} on a 2n-dimensional manifold $M$ is a pair $(\omega,\Omega)$ of a 1-form $\omega$ and a 2-form $\Omega$ which satisfies
  \[\Omega^n\neq0,\qquad d\omega=0,\qquad d\Omega=\omega\wedge\Omega.\]
  On an LCS manifold $(M,\omega,\Omega)$, there is a unique vector field $E$ and a unique bivector field $\pi$ such that
  \[\iota_E\Omega=-\omega,\qquad\iota_{\pi^\sharp\alpha}\Omega=-\alpha\quad(\alpha\in T^\ast M),\]
  where $\pi^\sharp:T^\ast M\to TM$ is the induced map by the interior product. Then the pair $(\pi,E)$ is a Jacobi structure on $M$.\End
\end{example}

For a Jacobi manifold $(M,\pi,E)$, a distribution given by $T\mathcal{F}_{(\pi,E)}=\textrm{Im}\pi^\sharp+\mathbb{R}\langle E\rangle$ is integrable, and thus defines a (possibly singular) foliation $\mathcal{F}_{(\pi,E)}$ on $M$. If $E_x \in \textrm{Im}\pi^\sharp_x$ holds for a point $x\in M$, a leaf through $x$ of the foliation $\mathcal{F}_{(\pi,E)}$ is even-dimensional and an LCS manifold. On the other hand, the leaf through $x$ is odd-dimensional and a contact manifold if $E_x \notin \textrm{Im}\pi^\sharp_x$. A Jacobi manifold $(M,\pi,E)$ is said to be \textit{regular} if the foliation $\mathcal{F}_{(\pi,E)}$ is regular. Regular Jacobi manifolds are classified into two cases:

\begin{itemize}
  \item The case where each leaf of $\mathcal{F}_{(\pi,E)}$ is even-dimensional (LCS manifold).
  \item The case where each leaf of $\mathcal{F}_{(\pi,E)}$ is odd-dimensional (contact manifold).
\end{itemize} We refer to the former case as an \textit{LCS-type regular Jacobi manifold} and the latter case as a \textit{contact-type regular Jacobi manifold}.

\vspace{0.1in}
Let $(M_1,\{\cdot,\cdot\}_1),(M_2,\{\cdot,\cdot\}_2)$ be Jacobi manifolds and $J:M_1\to M_2$ a smooth map. Then $J$ is said to be a \textit{conformal Jacobi map} if there is a non-zero function $a\in C^\infty(M_1)$ such that
\[\{f,g\}_2\circ J=\frac{1}{a}\{a\cdot(f\circ J),a\cdot(g\circ J)\}_1\]
holds for all $f,g\in C^\infty(M_2)$. If $a=1$, $J$ is called a \textit{Jacobi map}. When $(\pi_1,E_1),(\pi_2,E_2)$ are the corresponding tensors of Jacobi manifolds $(M_1,\{\cdot,\cdot\}_1),(M_2,\{\cdot,\cdot\}_2)$ respectively, a map $J:M_1\to M_2$ is a conformal Jacobi map if and only if 
\[J_\ast\pi_1=a\pi_2,\qquad J_\ast E_1=aE_2+\pi_2^\sharp(da)\]
holds for some non-zero function $a\in C^\infty(M_2)$. Two Jacobi manifolds $M_1,M_2$ are said to be \textit{equivalent (resp. conformally equivalent)} if there is a Jacobi map (resp. conformal Jacobi map) $J:M_1\to M_2$ which is diffeomorphic.

\section{Godbillon-Vey classes of regular foliations}

Let $M$ be a $n$-dimensional manifold and $\mathcal{F}$ a coorientable regular foliation of codimension $q\ (0<q<n)$ on $M$. Then there is a $q$-form $\alpha$ which is locally decomposable and satisfies $d\alpha=\beta\wedge\alpha$ for some 1-form $\beta$, such that the distribution $T\mathcal{F}$ is characterized by
\[T\mathcal{F}=\{X\in\mathfrak{X}(M)\mid \omega(X)=0,\ \omega\in\mathcal{I}\},\]
where $\mathcal{I}$ is defined by
\[\mathcal{I}=\{\omega\in\Omega^1(M)\mid \omega\wedge\alpha=0\}.\]
Since $q\neq n$, we have $\alpha\neq0$.

The forms $\alpha$ and $\beta$ (we call a pair $(\alpha,\beta)$ a \textit{defining pair} of $\mathcal{F}$) are not uniquely determined by $\mathcal{F}$ . For example, in the case of $q=1$, if $(\alpha,\beta)$ and $(\alpha',\beta')$ are two defining pairs of $\mathcal{F}$, there are a function $u\in C^\infty(M)$ and a non-zero function $v\in C^\infty(M)$ such that
\[\alpha'=v\alpha,\qquad\beta'=\beta-d\log|v|+uv\alpha\]
holds. However, we can define an invariant of $\mathcal{F}$ by using them:

\begin{theorem}[\cite{godbillon1971invariant}]
  The cohomology class $[\beta\wedge (d\beta)^q]\in H^{2q+1}_{dR}(M)$ is independent on the choice of the forms $\alpha,\beta$, and thus uniquely determined by the foliation $\mathcal{F}$.\hfill{$\Box$}
\end{theorem}
The cohomology class $[\beta\wedge (d\beta)^q]\in H^{2q+1}_{dR}(M)$ is called the \textit{Godbillon-Vey class} of the foliation $\mathcal{F}$ and denoted by $GV(\mathcal{F})$. About a proof of the theorem above, see \cite{moerdijk2003introduction} for example.

\vspace{0.1in}
Let $N,M$ be manifolds, $f:N\to M$ a smooth map, and $\mathcal{F}$ a regular foliation on $M$ of codimension $q$. Assume that the map $f$ is transversal to $\mathcal{F}$, namely, 
\[(df)_x(T_xN)+T_{f(x)}\mathcal{F}=T_{f(x)}M\]
holds for any $x\in N$. Then we obtain a regular foliation $f^\ast\mathcal{F}$ on $N$, which is called the \textit{pullback foliation} of $\mathcal{F}$ by $f$ (see \cite{moerdijk2003introduction} for details). The pullback foliation satisfies $\textrm{codim}f^\ast\mathcal{F}=q$ and $T(f^\ast\mathcal{F})=(f_\ast)^{-1}(T\mathcal{F})$.

For a coorientable regular foliation $\mathcal{F}$, let $(\alpha,\beta)$ be a defining pair. Then a pair $(f^\ast\alpha,f^\ast\beta)$ is a defining pair of $f^\ast\mathcal{F}$ (and thus $f^\ast\mathcal{F}$ is also coorientable). Therefore the Godbillon-Vey classes of $\mathcal{F}$ and $f^\ast\mathcal{F}$ satisfies the naturality condition:
\[GV(f^\ast\mathcal{F})=f^\ast GV(\mathcal{F}).\]

\begin{lemma}\label{lem:2}
  Let $(A,B)$ be a defining pair of the pullback foliation $f^\ast\mathcal{F}$ consists of $A\in\Omega^q(N)$ and $B\in\Omega^1(N)$. Assume that $f$ is a surjective submersion and a form $\eta\in\Omega^1(M)$ satisfies $f^\ast\eta=B$. Then the form $\eta$ gives the Godbillon-Vey class of $\mathcal{F}$, namely, $GV(\mathcal{F})=[\eta\wedge(d\eta)^q]$ holds.
\end{lemma}
\begin{proof}
  For simplicity, we assume that $q=1$ (we can prove for general $q$ as the same way). Let $(\alpha,\beta)$ be any defining pair of $\mathcal{F}$ consists of $\alpha\in\Omega^1(M)$ and $\beta\in\Omega^1(M)$. Then $(f^\ast\alpha,f^\ast\beta)$ is a defining pair of $f^\ast\mathcal{F}$, and thus we have
  \begin{equation}\label{eq:-1.5}
    f^\ast\eta=B=f^\ast\beta-d\log|\widetilde{v}|+\widetilde{u}\widetilde{v}f^\ast\alpha
  \end{equation}
  for some function $\widetilde{u}\in C^\infty(N)$ and some non-zero function $\widetilde{v}\in C^\infty(N)$. 
  
  From (\ref{eq:-1.5}), we can see that the functions $\widetilde{u},\widetilde{v}$ are basic, namely, there are a function $u\in C^\infty(M)$ and a non-zero function $v\in C^\infty(M)$ such that $\widetilde{u}=f^\ast u,\ \widetilde{v}=f^\ast v.$ In fact, for a vector field $X$ which is along each leaf, we have
  \[0=(f^\ast\eta)(X)-(f^\ast\beta)(X)=-\dfrac{d\widetilde{v}(X)}{\widetilde{v}}+\widetilde{u}\widetilde{v}(f^\ast\alpha)(X)=-\dfrac{d\widetilde{v}(X)}{\widetilde{v}}\]
  and thus $\widetilde{v}$ is constant on each leaf, namely, $\widetilde{v}$ is basic. Moreover, taking the Lie derivative of both sides of (\ref{eq:-1.5}) by a vector field $X$ which is along each leaf, we obtain
  \[(\mathcal{L}_X\widetilde{u})\widetilde{v}f^\ast\alpha+\widetilde{u}(\mathcal{L}_X\widetilde{v})f^\ast\alpha+\widetilde{u}\widetilde{v}\mathcal{L}_Xf^\ast\alpha=0\]
  and this implies $\mathcal{L}_X\widetilde{u}=0$ since $\widetilde{v}\neq0$ and $\alpha\neq0$. Therefore $\widetilde{u}$ is also basic.

  Then since $f$ is a surjective submersion, we obtain
  \[\eta=\beta-d\log|v|+uv\alpha,\]
  and thus $(v\alpha,\eta)$ is a defining pair of $\mathcal{F}$.
\end{proof}

Let us briefly examine the foliation determined by a Jacobi manifold. We can observe the following.

\begin{proposition}
  Let $(M_1,\pi_1,E_2)$ and $(M_2,\pi_2,E_2)$ be Jacobi manifolds, and $J:M_1\to M_2$ a conformally equivalence map. If one of $(M_1,\pi_1,E_2)$ and $(M_2,\pi_2,E_2)$ is regular, then so is the other. Moreover, $J^\ast\mathcal{F}_{(\pi_2,E_2)}=\mathcal{F}_{(\pi_1,E_1)}$ holds.
\end{proposition}
\begin{proof}
  Since $J$ is a diffeomorphism, the relation
$J_* \pi_1 = a \pi_2$ with $a \in C^\infty(M_1)$ a non-zero function implies that $J_* X \in \operatorname{Im}\pi_2^\sharp$ for any $X \in \operatorname{Im} \pi_1^\sharp$. Combining this with $J_* E_1 = a E_2 + \pi_2^\sharp(da)$, we obtain $J_*\bigl(T\mathcal{F}_{(\pi_1,E_1)}\bigr)\subset T\mathcal{F}_{(\pi_2,E_2)}$. Applying the same argument to $J^{-1}$, we conclude that $J_*\bigl(T\mathcal{F}_{(\pi_1,E_1)}\bigr)=T\mathcal{F}_{(\pi_2,E_2)}$.
\end{proof}

\begin{corollary}\label{cor:-1}
 In the same setting as in the previous proposition, if one of $\mathcal{F}_{(\pi_1,E_1)}$ and $\mathcal{F}_{(\pi_2,E_2)}$ is coorientable, then so is the other. Moreover, $GV(\mathcal{F}_{(\pi_1,E_1)})=J^\ast GV(\mathcal{F}_{(\pi_2,E_2)})$ holds.\hfill{$\Box$}
\end{corollary}

\begin{example}
  Mikami \cite{mikami2000godbillon} constructed a Poisson structure such that the Godbillon-Vey class is nontrivial on a compact manifold obtained as a quotient by a discrete subgroup of $PSL(2,\mathbb{R})$. By \autoref{cor:-1}, it follows that, by applying a conformal transformation (the case that the identity map is a conformally equivalence map) to this Poisson structure, one obtains a Jacobi structure on the same manifold whose Godbillon-Vey class is also nontrivial.\End
\end{example}

\section{Main results}
\subsection{Notations}
Hereafter, let $M$ be an orientable $n$-dimensional manifold. We fix a volume form ${\rm{vol}}_M$ on $M$. For non-zero multivector field $U\in\mathfrak{X}^k(M)$, we take a (not unique) multivector field $\ast U\in\mathfrak{X}^{n-k}(M)$ which satisfies $\textrm{vol}_M(U\wedge\ast U)=1$. We define an isomorphism between vector bundles $\varphi:\mathfrak{X}^k(M)\to\Omega^{n-k}(M)$ by $\varphi(X)=\iota_X{\rm{vol}}_M$ for each $0\leq k\leq n$. Then we have
\begin{equation}\label{eq:-1}
  \varphi^{-1}(\alpha)=(-1)^{k(n+1)}\iota_\alpha\varphi^{-1}(1)
\end{equation}
for $\alpha\in\Omega^k(M)$, and
\begin{equation}\label{eq:0}
  \varphi^{-1}(\varphi(U)\wedge\varphi(V))=(-1)^{(n+k)(l+1)}\iota_{\varphi(U)}V=(-1)^{(n+1)(n+l)}\iota_{\varphi(V)}U
\end{equation}
for $U\in\mathfrak{X}^k(M)$ and $V\in\mathfrak{X}^l(M)$.

\vspace{0.1in}
We define a map $\psi:\mathfrak{X}^k(M)\to\mathfrak{X}^{k-1}(M)$ by $\psi(U)=\varphi^{-1}d\varphi(U)$.

\begin{lemma}[\cite{mikami2000godbillon}]\label{lem:1}
  For any $U\in\mathfrak{X}^k(M)$ and $V\in\mathfrak{X}^l(M)$,
  \[\psi(U\wedge V)=(-1)^l[U,V]_S+(-1)^l\psi(U)\wedge V+U\wedge\psi(V)\]
  holds.
\end{lemma}
\begin{proof}
  If $V\in\mathfrak{X}(M)$, we have
  \[\begin{split}
    \mathcal{L}_V\iota_U\vol_M&=\iota_{[V,U]_S}\vol_M+\iota_U\mathcal{L}_V\vol_M\\
    &=\iota_{[V,U]_S}\vol_M+\iota_Ud\iota_V\vol_M\\
    &=\iota_{[V,U]_S}\vol_M+\iota_Ud\varphi(V)\\
    &=\iota_{[V,U]_S}\vol_M+\iota_U\varphi(\psi(V))\\
    &=\iota_{[V,U]_S}\vol_M+\iota_U\iota_{\psi(V)}\vol_M\\
    &=\iota_{[V,U]_S}\vol_M+\iota_{\psi(V)U}\vol_M,
\end{split}\]
and thus we obtain
  \[\begin{split}
        \psi(U\wedge V)&=\varphi^{-1}d\varphi(U\wedge V)=\varphi^{-1}d(\iota_{U\wedge V}\vol_M)\\
        &=\varphi^{-1}d(\iota_V\iota_U\vol_M)=\varphi^{-1}(\mathcal{L}_V-\iota_Vd)(\iota_U\vol_M)\\
        &=\varphi^{-1}\mathcal{L}_V(\iota_U\vol_M)-\varphi^{-1}\iota_Vd(\iota_U\vol_M)\\
        &=\varphi^{-1}(\iota_{[V,U]_S}\vol_M+\iota_{\psi(V)U}\vol_M)-\varphi^{-1}\iota_V\varphi\varphi^{-1}d\varphi(U)\\
        &=\varphi^{-1}(\iota_{[V,U]_S}\vol_M+\iota_{\psi(V)U}\vol_M)-\varphi^{-1}\iota_{\psi(U)\wedge V}\vol_M\\
        &=[V,U]_S+\psi(V)U-\psi(U)\wedge V.
    \end{split}\]
  By using the induction with respect to $l$, we can complete the proof.
\end{proof}

\begin{corollary}\label{cor:0}
  For any Jacobi structure $(\pi,E)$ and $k\in\mathbb{Z}_{\geq1}$, we have
  \[\psi(\pi^k)=k\psi(\pi)\wedge\pi^{k-1}+k(k-1)E\wedge\pi^{k-1},\]
  \[\psi(\pi^k\wedge E)=-k\psi(\pi)\wedge\pi^{k-1}\wedge E+\pi^k\wedge\psi(E).\]
\end{corollary}
\begin{proof}
  First, we have
  \[\begin{split}
        [\pi,\pi^k]_S&=[\pi,\pi\wedge\pi^{k-1}]_S=[\pi,\pi]_S\wedge\pi^{k-1}+\pi\wedge[\pi,\pi^{k-1}]_S\\
        &=2[\pi,\pi]_S\wedge\pi^{k-1}+\pi^2\wedge[\pi,\pi^{k-2}]_S\\
        &=\cdots\\
        &=(k-1)[\pi,\pi]_S\wedge\pi^{k-1}+\pi^{k-1}\wedge[\pi,\pi]_S\\
        &=k[\pi,\pi]_S\wedge\pi^{k-1}=2kE\wedge\pi^k.
    \end{split}\]
  Then from \autoref{lem:1} we obtain
  \[\begin{split}
    \psi(\pi^k)&=\psi(\pi^{k-1}\wedge\pi)\\
        &=[\pi^{k-1},\pi]_S+\psi(\pi^{k-1})\wedge\pi+\pi^{k-1}\wedge\psi(\pi)\\
        &=2(k-1)E\wedge\pi^{k-1}+\psi(\pi^{k-1})\wedge\pi+\pi^{k-1}\wedge\psi(\pi)\\
        &=2(k-1)E\wedge\pi^{k-1}+2(k-2)E\wedge\pi^{k-1}+\psi(\pi^{k-2})\wedge\pi^2+2\pi^{k-1}\wedge\psi(\pi)\\
        &=\cdots\\
        &=2\left(\sum_{i=1}^{k-1}i\right)E\wedge\pi^{k-1}+k\psi(\pi)\wedge\pi^{k-1}\\
        &=k\psi(\pi)\wedge\pi^{k-1}+k(k-1)E\wedge\pi^{k-1}.
    \end{split}\]
    Moreover,
    \[\begin{split}
      \psi(\pi^k\wedge E)&=-[\pi^k,E]_S-\psi(\pi^k)\wedge E+\psi(E)\pi^k\\
      &=-\psi(\pi^k)\wedge E+\psi(E)\pi^k\\
      &=-k\psi(\pi)\wedge\pi^{k-1}\wedge E+\psi(E)\pi^k
  \end{split}\]
  holds.
\end{proof}

For a Jacobi manifold $(M,\pi,E)$, we denote by $m$ a integer such that $\pi^m\neq0$ and $\pi^{m+1}=0$ holds, and by $q$ the codimension of $\mathcal{F}_{(\pi,E)}$. Then $q=n-2m$ holds for LCS-type regular Jacobi manifolds, and $q=n-2m-1$ holds for contact-type regular Jacobi manifolds. Hereafter we assume that $0<q<n$.

\subsection{Godbillon-Vey classes of LCS-type regular Jacobi manifolds}

Now we prove our main result for LCS-type regular Jacobi manifolds.

\begin{theorem}\label{thm:100}
  For an LCS-type regular Jacobi manifold $(M,\pi,E)$, we can take
    \[\alpha=\varphi\left(\frac{\pi^m}{m!}\right),\qquad\beta=\varphi[-\ast\pi^m,\pi^m]_S\]
    (so $\mathcal{F}_{(\pi,E)}$ is co-orientable), and 
    \[GV(\mathcal{F}_{(\pi,E)})=\big[\varphi(\iota_{(d\varphi[-\ast\pi^m,\pi^m]_S)^{q}}[-\ast\pi^m,\pi^m]_S)\big]\]
    holds.
\end{theorem}
\begin{proof}

By the local structure theorem for Jacobi manifolds \cite{dazord1991structure}, for any $x\in M$, there is a neighborhood $U\ni x$ such that $(U,\pi|_U,E|_U)$ is conformally equivalent to a Jacobi manifold 
\[\mathbb{R}^{2m}\times\mathbb{R}^{q}=\{(x_1,\cdots,x_{2m},y_1,\cdots,y_q)\},\] 
        \[\pi_{\mathbb{R}^{2m}\times\mathbb{R}^{q}}=\sum_{i=1}^{m}\frac{\partial}{\partial x_{2i-1}}\wedge\frac{\partial}{\partial x_{2i}},\quad E_{\mathbb{R}^{2m}\times\mathbb{R}^{q}}=0.\]
We regard $X_i:=\frac{\partial}{\partial x_i}$ as a vector field on $U$ for $1\leq i\leq 2m$. Since we assume that $q\geq 1$, we can extend them to a local frame $\{X_1,X_2,\cdots, X_n\}$ such that $\vol_M(X_1\wedge\cdots\wedge X_n)=1$ holds. Then we can locally express $\pi$ as $\pi=a\sum_{i=1}^m X_{2i-1}\wedge X_{2i}$ for some non-zero function $a\in C^\infty(U)$, and thus we obtain $\pi^m=m!a^mX_1\wedge\cdots\wedge X_{2m}$.

Let $\{\theta^1,\cdots,\theta^n\}$ be the dual frame of $\{X_1,\cdots, X_n\}$. Then 
\[\alpha:=\varphi\left(\frac{\pi^m}{m!}\right)=a^m\theta^{2m+1}\wedge\cdots\wedge\theta^n\]
holds. Therefore we have
\[\mathcal{I}:=\{\omega\in\Omega^1(M)\mid \omega\wedge\alpha=0\}={\rm span}\{\theta^{2m+1},\cdots,\theta^n\}\]
and
\[\{X\in\mathfrak{X}(M)\mid \omega(X)=0,\ \omega\in\mathcal{I}\}={\rm span}\{X_1,\cdots,X_{2m}\}={\rm Im}\pi^\sharp=\mathcal{F}_{(\pi,E)}\]
(note that $E=\frac{1}{a}\pi^\sharp(da)$ holds on $U$).

\vspace{0.1in}
Next we prove that the 1-form $\beta:=\varphi[-\ast\pi^m,\pi^m]_S$ satisfies $d\alpha=\beta\wedge\alpha$. We have
\[\beta\wedge\alpha=(-1)^{q-1}\alpha\wedge(\iota_{[\ast\pi^m,\pi^m]_S}\vol_M)=-\iota_{\iota_\alpha[-\ast\pi^m,\pi^m]_S}\vol_M,\]
and
\[d\alpha=d\varphi\left(\frac{\pi^m}{m!}\right)=\varphi\psi\left(\frac{\pi^m}{m!}\right)=\iota_{\psi\left(\frac{\pi^m}{m!}\right)}\vol_M,\]
hence it remains to show that
\[-\iota_\alpha[\ast\pi^m,\pi^m]_S=\psi\left(\frac{\pi^m}{m!}\right).\]
From \autoref{lem:1}, we have
\[\begin{split}
    0&=\varphi^{-1}d\iota_{\ast\pi^m\wedge\pi^m}\vol_M=\psi(\ast\pi^m\wedge\pi^m)\\
    &=[\ast\pi^m,\pi^m]_S+\psi(\ast\pi^m)\wedge\pi^m+\ast\pi^m\wedge\psi(\pi^m),
\end{split}\]
and thus
\[-\iota_\alpha[\ast\pi^m,\pi^m]_S=\iota_\alpha(\psi(\ast\pi^m)\wedge\pi^m)+\iota_\alpha(\ast\pi^m\wedge\psi(\pi^m))\]
holds. Then we show that 
\begin{equation}\label{eq:1}
  \iota_\alpha(\psi(\ast\pi^m)\wedge\pi^m)=0
\end{equation}
and
\begin{equation}\label{eq:2}
  \iota_\alpha(\ast\pi^m\wedge\psi(\pi^m))=\psi\left(\frac{\pi^m}{m!}\right).
\end{equation}

First, we can locally express $\pi^m$ as $\pi^m=m!a^m X_1\wedge\cdots\wedge X_{2m}$. Thus for any $U\in\mathfrak{X}^{q-1}(M)$ and $1\leq i\leq 2m$, we have $\alpha(X_i\wedge U)=\vol_M(\frac{\pi^m}{m!}\wedge X_i\wedge U)=0$, which implies (\ref{eq:1}) since $\alpha\in\Omega^q(M)$ and $\psi(\ast\pi^m)\in\mathfrak{X}^{q-1}(M)$.

\vspace{0.1in}
Next, we prove (\ref{eq:2}). From \autoref{cor:0} we have
\[0=\psi(\pi^{m+1})=(m+1)\psi(\pi)\wedge\pi^m+m(m+1)E\wedge\pi^m.\]
Moreover, $E\wedge\pi^m=0$ holds since $E\in{\rm Im}\pi^\sharp$. Thus we have $\psi(\pi)\wedge\pi^m=0$, and then
\[\iota_{\psi(\pi)}\alpha=\iota_{\psi(\pi)}\iota_{\frac{\pi^m}{m!}}\vol_M=0\]
holds. Therefore we have $(\ast\pi^m)\wedge\iota_\alpha(\psi(\pi^m))=0$ by \autoref{cor:0} and obtain
\[\iota_\alpha(\ast\pi^m\wedge\psi(\pi^m))=\alpha(\ast\pi^m)\psi(\pi^m)=\frac{\psi(\pi^m)}{m!}.\]

From the above, it can be concluded that the form $\beta\wedge(d\beta)^q$ represents the Godbillon-Vey class of $\mathcal{F}_{(\pi,E)}$, and from (\ref{eq:-1}) the form $\beta\wedge(d\beta)^q$ is described as
\[\begin{split}
    \beta\wedge(d\beta)^q&=\varphi\varphi^{-1}(\beta\wedge(d\beta)^q)=\varphi((-1)^{(n+1)(2q+1)}\iota_{\beta\wedge(d\beta)^q}\varphi^{-1}(1))\\
    &=\varphi(\iota_{(d\beta)^q}\varphi^{-1}(\beta))=\varphi(\iota_{(d\varphi[-\ast\pi^m,\pi^m]_S)^{q}}[-\ast\pi^m,\pi^m]_S).
\end{split}\]
\end{proof}

\subsection{Godbillon-Vey classes of contact-type regular Jacobi manifolds}

Next we prove for contact-type regular Jacobi manifolds.

\begin{theorem}\label{thm:101}
  For a contact-type regular Jacobi manifold $(M,\pi,E)$, we can take 
  \[\alpha=\varphi\left(\frac{\pi^m}{m!}\wedge E\right),\qquad\beta=\varphi[\ast(\pi^m\wedge E),\pi^m\wedge E]_S\]
  (so $\mathcal{F}_{(\pi,E)}$ is co-orientable), and
  \[GV(\mathcal{F}_{(\pi,E)})=\big[\varphi(\iota_{(d\varphi[\ast(\pi^m\wedge E),\pi^m\wedge E]_S)^{q}}[\ast(\pi^m\wedge E),\pi^m\wedge E]_S)\big]\]
  holds.
\end{theorem}
\begin{proof}
  By the local structure theorem for Jacobi manifolds \cite{dazord1991structure}, for any $x\in M$, there is a neighborhood $U\ni x$ such that $(U,\pi|_U,E|_U)$ is equivalent to a Jacobi manifold 

  \[\mathbb{R}^{2m+1}\times\mathbb{R}^{q}=\{(x_0,x_1,\cdots,x_{2m},y_1,\cdots,y_q)\},\] 
        \[\pi_{\mathbb{R}^{2m+1}\times\mathbb{R}^{q}}=\sum_{i=1}^{m}\left(\frac{\partial}{\partial x_{2i-1}}-x_{2i}\frac{\partial}{\partial x_0}\right)\wedge\frac{\partial}{\partial x_{2i}},\quad E_{\mathbb{R}^{2m+1}\times\mathbb{R}^{q}}=\frac{\partial}{\partial x_0}.\]

  We regard $X_i:=\frac{\partial}{\partial x_i}$ as a vector field on $U$ for $0\leq i\leq 2m$. Since we assume that $q\geq 1$, we can extend them to a local frame $\{X_0,X_1,\cdots, X_{n-1}\}$ such that $\vol_M(X_0\wedge\cdots\wedge X_{n-1})=1$ holds. Then we can locally express $\pi^m\wedge E$ as $\pi^m\wedge E=m!X_0\wedge\cdots\wedge X_{2m}$.

  Let $\{\theta^0,\theta^1,\cdots,\theta^n\}$ be the dual frame of $\{X_0,X_1,\cdots, X_{n-1}\}$. Then 
  \[\alpha:=\varphi\left(\frac{\pi^m}{m!}\wedge E\right)=\theta^{2m+1}\wedge\cdots\wedge\theta^n\]
  holds. Therefore we have
  \[\mathcal{I}:=\{\omega\in\Omega^1(M)\mid \omega\wedge\alpha=0\}={\rm span}\{\theta^{2m+1},\cdots,\theta^n\}\]
  and
  \[\{X\in\mathfrak{X}(M)\mid \omega(X)=0,\ \omega\in\mathcal{I}\}={\rm span}\{X_0,\cdots,X_{2m}\}=\textrm{Im}\pi^\sharp+\mathbb{R}\langle E\rangle=\mathcal{F}_{(\pi,E)}.\]

\vspace{0.1in}
Next we prove that the 1-form $\beta:=\varphi[\ast(\pi^m\wedge E),\pi^m\wedge E]_S$ satisfies $d\alpha=\beta\wedge\alpha$. Similarly to the proof of \autoref{thm:100}, it remains to show that
\[\iota_\alpha[\ast(\pi^m\wedge E),\pi^m\wedge E]_S=\psi\left(\frac{\pi^m}{m!}\wedge E\right).\]
From \autoref{lem:1}, we have
\[\begin{split}
    0&=\varphi^{-1}d\iota_{\ast(\pi^m\wedge E)\wedge(\pi^m\wedge E)}\vol_M=\psi(\ast(\pi^m\wedge E)\wedge(\pi^m\wedge E))\\
    &=-[\ast(\pi^m\wedge E),\pi^m\wedge E]_S-\psi(\ast(\pi^m\wedge E))\wedge(\pi^m\wedge E)+\ast(\pi^m\wedge E)\wedge\psi(\pi^m\wedge E),
\end{split}\]
and thus
\[\iota_\alpha[\ast(\pi^m\wedge E),\pi^m\wedge E]_S=-\iota_\alpha(\psi(\ast(\pi^m\wedge E))\wedge(\pi^m\wedge E))+\iota_\alpha(\ast(\pi^m\wedge E)\wedge\psi(\pi^m\wedge E))\]
holds. Then we show that 
\begin{equation}\label{eq:3}
  \iota_\alpha(\psi(\ast(\pi^m\wedge E))\wedge(\pi^m\wedge E))=0
\end{equation}
and
\begin{equation}\label{eq:4}
  \iota_\alpha(\ast(\pi^m\wedge E)\wedge\psi(\pi^m\wedge E))=\psi\left(\frac{\pi^m}{m!}\wedge E\right).
\end{equation}

First, we can locally express $\pi^m\wedge E$ as $\pi^m\wedge E=m!X_0\wedge\cdots\wedge X_{2m}$. Thus for any $U\in\mathfrak{X}^{q-1}(M)$ and $0\leq i\leq 2m$, we have $\alpha(X_i\wedge U)=\vol_M(\frac{\pi^m}{m!}\wedge E\wedge X_i\wedge U)=0$, which implies (\ref{eq:3}) since $\alpha\in\Omega^q(M)$ and $\psi(\ast(\pi^m\wedge E))\in\mathfrak{X}^{q-1}(M)$.

\vspace{0.1in}
Next, we prove (\ref{eq:4}). From \autoref{cor:0} we have
\[0=\psi(\pi^{m+1}\wedge E)=-(m+1)\psi(\pi)\wedge\pi^m\wedge E+\psi(E)\pi^{m+1}=-(m+1)\psi(\pi)\wedge\pi^m\wedge E.\]
Thus we obtain $\psi(\pi)\wedge\pi^m\wedge E=0$, and then
\[\iota_{\psi(\pi)}\alpha=\iota_{\psi(\pi)}\iota_{\frac{\pi^m}{m!}\wedge E}\vol_M=0\]
holds. Therefore we have $\ast(\pi^m\wedge E)\wedge\iota_\alpha(\psi(\pi^m\wedge E))=0$ by \autoref{cor:0} and obtain

\[\iota_\alpha(\ast(\pi^m\wedge E)\wedge\psi(\pi^m\wedge E))=\alpha(\ast(\pi^m\wedge E))\psi(\pi^m\wedge E)=\frac{\psi(\pi^m\wedge E)}{m!}.\]

\end{proof}

\section{Viewpoint of Poissonizations}

In the case of contact-type regular Jacobi manifolds, \autoref{thm:101} can be derived by an alternative method, namely, using Poissonization. For any Jacobi structure $(\pi,E)$ on $M$, we can associate a Poisson structure $\Lambda$ on $\widetilde{M}:=M\times\mathbb{R}_{>0}$ defined by
\[\Lambda=\frac{1}{t}\pi+\frac{\partial}{\partial t}\wedge E.\]
The Poisson manifold $(\widetilde{M},\Lambda)$ is called the \textit{Poissonization} of $(M,\pi,E)$. For a regular Jacobi manifold, its Poissonization is a regular Poisson manifold. In the following, $n$ denotes the dimension of $M$.

\begin{lemma}
  For a contact-type regular Jacobi manifold $(M,\pi,E)$, the pullback foliation of the foliation $\mathcal{F}_{(\pi,E)}$ by the projection ${\rm pr}:\widetilde{M}\to M$ is the symplectic foliation of the Poisson manifold $(\widetilde{M},\Lambda)$.
\end{lemma}
\begin{proof}
  We can see that
  \[\Lambda^{m+1}=(m+1)\frac{1}{t^m}\pi^m\wedge\partialt\wedge E\]
  and $\Lambda^{m+1}\neq0,\ \Lambda^{m+2}=0.$ Thus $\rank\Lambda^\sharp=2m+2$ holds. On the other hand, we have $\rank T({\rm pr}^\ast\mathcal{F}_{(\pi,E)})=\rank T\mathcal{F}_{(\pi,E)}+1=2m+2$. Moreover, using a coordinate $((x_i),t)$ on $\widetilde{M}$ we have
  \[\begin{split}
    \Lambda^\sharp(df)&=\Lambda^\sharp\left(\displaystyle\sum_i\frac{\partial f}{\partial x_i}dx_i+\frac{\partial f}{\partial t}dt\right)\\
    &=\frac{1}{t}\displaystyle\sum_i\frac{\partial f}{\partial x_i}\pi^\sharp(dx_i)+\frac{\partial f}{\partial t}E-\left(\displaystyle\sum_i\frac{\partial f}{\partial x_i}dx_i(E)\right)\partialt
\end{split}\]
  for $f\in C^\infty(\widetilde{M})$, thus ${\rm pr}_\ast({\rm Im}\Lambda^\sharp)\subset T\mathcal{F}_{(\pi,E)}$ holds. Since $T({\rm pr}^\ast\mathcal{F}_{(\pi,E)})=({\rm pr}_\ast)^{-1}(T\mathcal{F}_{(\pi,E)})$, we obtain $T({\rm pr}^\ast\mathcal{F}_{(\pi,E)})={\rm Im}\Lambda^\sharp$.
\end{proof}

\begin{remark}
  The lemma above does not hold for LCS-type regular Jacobi manifolds. In fact, if the pullback foliation of $\mathcal{F}_{(\pi,E)}$ of an LCS-type regular Jacobi manifold is the symplectic foliation, both $\rank\mathcal{F}_{(\pi,E)}$ and $\rank f^\ast\mathcal{F}_{(\pi,E)}$ must be even. However,
  \[\rank\mathcal{F}_{(\pi,E)}=n-\textrm{codim}\mathcal{F}_{(\pi,E)}=n-\textrm{codim}f^\ast\mathcal{F}_{(\pi,E)}=\rank f^\ast\mathcal{F}_{(\pi,E)}-1\]
  holds, that is a contradiction.\End
\end{remark}

Let $(M,\pi,E)$ be a contact-type regular Jacobi manifold such that ${\rm rank}T\mathcal{F}_{(\pi,E)}=2m+1$. We take a volume form $\vol_M\wedge dt$ on $\widetilde{M}$ and fix it. Then we can define $\Phi:\mathfrak{X}^k(\widetilde{M})\to\Omega^{n+1-k}(\widetilde{M})$ in the same way as $\varphi$, and for non-zero multivector field $U\in\mathfrak{X}^k(\widetilde{M})$ we take a (not unique) multivector field $\star U\in\mathfrak{X}^{n+1-k}(\widetilde{M})$ in the same way as $\ast U\in\mathfrak{X}^{n-k} (M)$ for $U\in\mathfrak{X}^k (M)$. From the result of \cite{mikami2000godbillon}, a defining pair $(A,B)$ of the symplectic foliation of $(\widetilde{M},\Lambda)$ is given by
\[A=\Phi\left(\dfrac{\Lambda^{m+1}}{(m+1)!}\right),\qquad B=\Phi[-\star\Lambda^{m+1},\Lambda^{m+1}]_S.\]

\begin{proposition}\label{prop:0}
  Let $\beta=\varphi[\ast(\pi^m\wedge E),\pi^m\wedge E]_S\in\Omega^1(M)$. Define $A'\in\Omega^q(\widetilde{M})$ and $B'\in\Omega^1(\widetilde{M})$ as
  \[A'=e^{t^m}A,\qquad B'={\rm pr}^\ast((-1)^q\beta).\]Then $(A',B')$ is a defining pair of the symplectic foliation of $(\widetilde{M},\Lambda)$.
\end{proposition}
\begin{proof}

By using (\ref{eq:-3}) and (\ref{eq:-2}) repeatedly, we obtain
\[\begin{split}
    [-\star\Lambda^{m+1},\Lambda^{m+1}]_S&=-\left[\star\left(\frac{1}{t^m}\pi^m\wedge E\wedge \partialt\right),\frac{1}{t^m}\pi^m\wedge E\wedge \partialt\right]_S\\
    &=(-1)^{q+1}\left[t^m\ast(\pi^m\wedge E),\frac{1}{t^m}\pi^m\wedge E\wedge \partialt\right]_S\\
    &=(-1)^{q+1}\left[t^m\ast(\pi^m\wedge E),\frac{1}{t^m}\right]_S\wedge\pi^m\wedge E\wedge\partialt\\
    &\qquad+(-1)^{q+1}\frac{1}{t^m}\left[t^m\ast(\pi^m\wedge E),\pi^m\wedge E\wedge\partialt\right]_S\\
    &=(-1)^{q+1}\left[\ast(\pi^m\wedge E),\pi^m\wedge E\wedge \partialt\right]_S-\left[t^m,\pi^m\wedge E\wedge\partialt\right]_S\wedge\ast(\pi^m\wedge E)\\
    &=(-1)^{q+1}\left[\ast(\pi^m\wedge E),\pi^m\wedge E\right]_S\wedge \partialt-mt^{m-1}\pi^m\wedge E\wedge \ast(\pi^m\wedge E).
\end{split}\]

\noindent Hence we have
\[\begin{split}
    B&=\Phi[-\star\Lambda^{m+1},\Lambda^{m+1}]_S\\
    &=(-1)^q\varphi[\ast(\pi^m\wedge E),\pi^m\wedge E]_S-mt^{m-1}dt\\
    &=(-1)^q\textrm{pr}^\ast\beta-mt^{m-1}dt.
\end{split}\]

\noindent Therefore
\[\begin{split}
    B'\wedge A'&=(B+mt^{m-1}dt)\wedge e^{t^m}A\\
    &=e^{t^m}B\wedge A+mt^{m-1}e^{t^m}dt\wedge A\\
    &=e^{t^m}dA+mt^{m-1}e^{t^m}dt\wedge A=dA'
\end{split}\]follows (note that $dA=B\wedge A$ holds since $(A,B)$ is a defining pair). Clearly $A'$ defines the same foliation as $A$, and thus $(A',B')$ is a defining pair of the foliation.
\end{proof}

\autoref{prop:0} and \autoref{lem:2} implies that a form $(-1)^q\beta\wedge(d(-1)^q\beta)^q=\beta\wedge(d\beta)^q$ represents $GV(\mathcal{F}_{(\pi,E)})$.

\section{Codimension one case}

In this section, we discuss the case of regular Jacobi manifolds with codimension one in detail. Such Jacobi manifolds are further classified into two cases: namely, LCS-type regular Jacobi structures of rank $2m$ on $(2m+1)$-dimensional manifolds, and contact-type regular Jacobi structures of rank $2m+1$ on $(2m+2)$-dimensional manifolds.

\begin{proposition}\label{prop:1}
  \ 
  \begin{itemize}
    \item For an LCS-type regular Jacobi manifold $(M,\pi,E)$ whose foliation $\mathcal{F}_{(\pi,E)}$ is of codimension one,
    \[GV(\mathcal{F}_{(\pi,E)})=\big[\varphi(\iota_{\varphi[-\ast\pi^{m},\pi^m]_S}{\psi}[-\ast\pi^{m},\pi^m]_S)\big]\]
    holds.
    \item For a contact-type regular Jacobi manifold $(M,\pi,E)$ whose foliation $\mathcal{F}_{(\pi,E)}$ is of codimension one,
    \[GV(\mathcal{F}_{(\pi,E)})=\big[\varphi(-\iota_{\varphi[\ast(\pi^{m}\wedge E),\pi^m\wedge E]_S}{\psi}[\ast(\pi^{m}\wedge E),\pi^m\wedge E]_S)\big]\]
    holds.
  \end{itemize}
\end{proposition}
\begin{proof}
  We prove only for LCS-type (we can prove also for contact-type in the same way). By (\ref{eq:-1}) we have
  \[\begin{split}
    \varphi^{-1}(\beta\wedge d\beta)&=(-1)^{n+1}\iota_{\beta}\iota_{d\beta}\varphi^{-1}(1)\\
    &=\iota_\beta(\varphi^{-1}d\varphi[-\ast\pi^{m},\pi^m]_S)=\iota_{\varphi[-\ast\pi^{m},\pi^m]_S}{\psi}[-\ast\pi^{m},\pi^m]_S.
\end{split}\]
\end{proof}

The following is a direct implication of \autoref{prop:1}.
\begin{corollary}
  Assume that $\mathcal{F}_{(\pi,E)}$ has codimension one. Then the following holds.
  \begin{itemize}
    \item In the case of LCS-type, we have $GV(\mathcal{F})_{(\pi,E)}=0$ if the vector field $\ast\pi^m$ satisfies $\mathcal{L}_{\ast\pi^m}\pi=0$.
    \item In the case of contact-type, we have $GV(\mathcal{F})_{(\pi,E)}=0$ if the vector field $\ast(\pi^m\wedge E)$ satisfies $\mathcal{L}_{\ast(\pi^m\wedge E)}\pi=0$ and $\mathcal{L}_{\ast(\pi^m\wedge E)}E=0$.\hfill{$\Box$}
  \end{itemize}
\end{corollary}

\begin{remark}
  For a manifold $M$ of dimension $n=3$ or $n=4$, the codimension $q$ of a regular Jacobi structure with a non-trivial Godbillon-Vey class is restricted to $q=1$ only. If $n=3$ and $q=1$, $M$ becomes an LCS-type regular Jacobi manifold. On the other hand, if $n=4$ and $q=1$, $M$ becomes a contact-type regular Jacobi manifold. \End
\end{remark}

  For a Poisson manifold $(M,\pi)$, a vector field $\psi(\pi)$ is called the \textit{modular vector field} of $(M,\pi)$ relative to the volume form $\vol_M$. $(M,\pi)$ is said to be \textit{unimodular} if $\psi(\pi)=0$ holds for some volume form. Mikami \cite{mikami2000godbillon} showed that if a Poisson manifold with codimension one symplectic foliation is unimodular, then the Godbillon-Vey class of the symplectic foliation vanishes. In what follows, we show a similar result for Jacobi manifolds.

\begin{corollary}\label{cor:1}
  Assume that $\mathcal{F}_{(\pi,E)}$ has codimension one. Then the following holds.
  \begin{itemize}
    \item If $(M,\pi,E)$ is a LCS-type regular Jacobi manifold and $\psi(\pi^m)=0$ holds, then we have $GV(\mathcal{F}_{(\pi,E)})=0$.
    \item If $(M,\pi,E)$ is a contact-type regular Jacobi manifold and $\psi(\pi)=0,\ \psi(E)=0$ holds, then we have $GV(\mathcal{F}_{(\pi,E)})=0$.
  \end{itemize}
\end{corollary}
\begin{proof}
  First, we prove for LCS-type. From \autoref{lem:1} and the assumption $\psi(\pi^m)=0$, we have

  \begin{equation}\label{eq:5}
    [-\ast\pi^m,\pi^m]_S=\psi(\pi^m)\wedge(-\ast\pi^m)-\psi(-\ast\pi^m)\pi^m=-\psi(-\ast\pi^m)\pi^m,
  \end{equation}and thus

  \begin{equation}\label{eq:6}
    \begin{split}
    \psi[-\ast\pi^m,\pi^m]_S&=-[\psi(-\ast\pi^m),\pi^m]_S-\psi^2(-\ast\pi^m)\pi^m-\psi(-\ast\pi^m)\wedge\psi(\pi^m)\\
    &=-[\psi(-\ast\pi^m),\pi^m]_S
  \end{split}
  \end{equation}holds. We locally express $\pi^m$ as $\pi^m=m!a^m X_1\wedge\cdots\wedge X_{2m}$. Then we have
  \[\varphi(\psi(-\ast\pi^m)\pi^m)=m!a^m\psi(-\ast\pi^m)\theta^{2m+1}\]
  and
  \[[\psi(-\ast\pi^m),\pi^m]_S=m!a^m\displaystyle\sum_{i=1}^{2m}(-1)^{i+1}[\psi(-\ast\pi^m),X_i]_SX_1\wedge\cdots\wedge\check{X_i}\wedge\cdots\wedge X_{2m}.\]

  \noindent Therefore we obtain
  \[\iota_{\varphi(\psi(-\ast\pi^m)\pi^m)}[\psi(-\ast\pi^m),\pi^m]_S=0.\]

  \noindent On the other hand, by (\ref{eq:5}) and (\ref{eq:6}), we have
  \[\begin{split}
    \varphi^{-1}(\beta\wedge d\beta)&=\iota_{\varphi[-\ast\pi^{m},\pi^m]_S}{\psi}[-\ast\pi^{m},\pi^m]_S\\
    &=\iota_{\varphi(\psi(-\ast\pi^m)\pi^m)}[\psi(-\ast\pi^m),\pi^m]_S,
  \end{split}\]and thus $GV(\mathcal{F}_{(\pi,E)})=0$ follows.

  \vspace{0.1in}
  Next, we prove for contact-type. As in the case of LCS-type, we have
  \[[\ast(\pi^m\wedge E),\pi^m\wedge E]_S=\psi(\pi^m\wedge E)\wedge\ast(\pi^m\wedge E)-\psi(\ast(\pi^m\wedge E))\pi^m\wedge E\]
  and
  \[\psi[\ast(\pi^m\wedge E),\pi^m\wedge E]_S=[\ast(\pi^m\wedge E),\psi(\pi^m\wedge E)]_S+[\psi(\ast(\pi^m\wedge E)),\pi^m\wedge E]_S.\]
  Under the assumption that $\psi(\pi)=0$ and $\psi(E)=0$, \autoref{cor:0} implies that $\psi(\pi^m\wedge E)=0$, and thus we obtain
  \[\varphi^{-1}(\beta\wedge d\beta)=\iota_{\varphi(\psi(\ast(\pi^m\wedge E))\pi^m\wedge E)}[\psi(\ast(\pi^m\wedge E)),\pi^m\wedge E]_S.\]
  Moreover,
  \[\iota_{\varphi(\psi(\ast(\pi^m\wedge E))\pi^m\wedge E)}[\psi(\ast(\pi^m\wedge E)),\pi^m\wedge E]_S=0\]
  holds similarly to the case of LCS-type.
\end{proof}

\begin{remark}
  In the Poisson case \cite{mikami2000godbillon}, the condition $\psi(\pi)=0$ automatically implies $\psi(\pi^m)=0$, so it suffices to assume $\psi(\pi)=0$. However, this is no longer the case for LCS-type regular Jacobi manifolds (see \autoref{cor:0}).\End
\end{remark}

\vspace{0.2in}
By defining a bracket on the space of sections of a line bundle, one can introduce the notion of a \textit{Jacobi bundle} \cite{marle1991jacobi}, which generalizes Jacobi manifolds. As in the case of Jacobi manifolds, a Jacobi bundle is known to induce a foliation. Extending the results of this paper to the foliation associated with a Jacobi bundle is a problem for future work.

\vspace{0.2in}
\bibliography{hoge} 
\bibliographystyle{alpha} 

\end{document}